\documentclass[letterpaper, 10 pt, conference]{ieeeconf}  
\IEEEoverridecommandlockouts                              
\usepackage{cite}
\usepackage{amsmath,amssymb,amsfonts,mathrsfs}
\usepackage{graphicx}
\usepackage{tikz,pgfplots}
\usetikzlibrary{arrows.meta, shapes.geometric, positioning, calc}
\pgfplotsset{compat=1.18}
\usepackage[hidelinks]{hyperref}

\newtheorem{theorem}{Theorem}
\newtheorem{proposition}{Proposition}
\newtheorem{lemma}{Lemma}

\newtheorem{definition}{Definition}
\newtheorem{assumption}{Assumption}
\newtheorem{remark}{Remark}
\newtheorem{example}{Example}

\newcommand{\X}{\mathbb{X}}
\newcommand{\TX}{\mathbb{T}\mathbb{X}}
\newcommand{\Q}{\mathbb{Q}}
\newcommand{\R}{\mathbb{R}}
\newcommand{\NE}{\mathbb{NE}}
\newcommand{\BR}{\mathbb{B}}
\newcommand{\PM}{\mathfrak{F}}
\newcommand{\cF}{\mathcal{F}}
\newcommand{\cG}{\mathcal{G}}
\newcommand{\cH}{\mathcal{H}}
\newcommand{\cV}{\mathcal{V}}
\newcommand{\cT}{\mathcal{T}}
\newcommand{\cS}{\mathcal{S}}
\newcommand{\cD}{\mathcal{D}}
\newcommand{\cL}{\mathcal{L}}
\newcommand{\Breg}{\mathcal{B}_{-g}}
\newcommand{\Bregf}{\mathcal{B}_{-f}}
\newcommand{\sX}{\mathscr{X}}
\newcommand{\sP}{\mathscr{P}}
\newcommand{\pbar}{\bar{p}}
\newcommand{\ptil}{\tilde{p}}
\newcommand{\fmax}{f_{\max}}
\newcommand{\Eq}{\mathbb{E}}
\newcommand{\dist}{\mathrm{dist}}
\DeclareMathOperator*{\argmax}{arg\,max}

\title{\LARGE \bf Learning in Large Populations With Anticipatory Payoffs: \\ Bregman Meets Willems and Lyapunov}

\author{Tzuyu Jeng and Nuno C. Martins
\thanks{This work was supported by NSF Grant ECCS2536055 and AFOSR Grant FA95502310467.}
\thanks{The authors are with the Department of Electrical and Computer Engineering and the Institute for Systems Research, University of Maryland, College Park, MD 20742, USA (e-mail: tyjeng@umd.edu, nmartins@umd.edu).}%
}

\begin{document}

\maketitle
\thispagestyle{empty}
\pagestyle{empty}

\begin{abstract}
We study evolutionary Nash equilibrium learning in large populations whose payoffs are generated by dynamic mechanisms. Counterclockwise dissipativity (CCW) provides a passivity-like framework for establishing learning/convergence under a broad convex cone of so-called well-behaved learning rules satisfying a natural payoff-alignment condition according to which, away from equilibria, the population's state (strategic profile) moves in a direction positively correlated with the payoffs. The cone is closed under conic combinations and encompasses a broad range of noise-free rules, including discontinuous rules such as best response, imitative rules such as the replicator rule that is not $\delta$-passive, and all noise-free $\delta$-passive rules. Existing CCW results require no knowledge of the particular rule, but establish convergence without providing a Lyapunov function or Lyapunov stability. We develop a methodology that closes this gap by constructing rule-independent Lyapunov functions from CCW storage functions of the payoff mechanism. For a nonlinear class of anticipatory mechanisms, in which agents compare current payoffs with those associated with a smoothed population state, we show that a storage can be derived as a Bregman divergence. The resulting Lyapunov function combines a potential gap with this divergence and establishes global asymptotic stability. A congestion game with a nonlinear congestion-growth toll illustrates the results.
\end{abstract}

\section{Introduction}
\label{sec:introduction}

Population games and evolutionary dynamics~\cite{Sandholm2010Population-Game,Sandholm2015Handbook-of-gam} provide systematic tools to model and analyze the strategic, non-cooperative interactions of large populations of learning agents. A learning rule (rule for short) describes how agents revise their strategies in response to the payoffs these strategies yield, and it may be programmed into artificial agents or model the bounded rationality of natural agents. Since the agents are indistinguishable, the aggregate strategic profile is captured by a population state, whose entries are the proportions of agents using each strategy, and which evolves according to an evolutionary dynamics model determined by the rule. The payoffs are produced by a payoff mechanism, which may represent the environment where the interaction takes place, as in congestion games~\cite{Beckmann1956Studies-in-the-}, or may be designed by a coordinator to steer the population towards desirable equilibria~\cite{Quijano2017The-role-of-pop,Martins2023Epidemic-popula,Certorio2022Epidemic-Popula,Barreiro-Gomez2016Constrained-dis}. 

We are interested in evolutionary Nash equilibrium learning, which refers to the convergence of the population state to the Nash equilibria set~\cite{Jr.1951Non-Cooperative} of the payoff mechanism, as opposed to Nash equilibrium seeking~\cite{Gadjov_2019aa}, which concerns processes not naturally expressed as evolutionary dynamics. When it holds, the Nash equilibria set predicts the long-term behavior of the population and, in potential games, has optimality properties~\cite{Sandholm2001Potential-games} that a designer can exploit.

\subsection{Passivity Approaches and the Gap We Address}

Early results on Nash equilibrium learning focused on memoryless payoff mechanisms, in particular potential games~\cite{Sandholm2001Potential-games} and contractive games~\cite{Hofbauer2009Stable-games-an}. More recently, dynamic payoff mechanisms have been studied through the system-theoretic notion of $\delta$-passivity~\cite{Fox2013Population-Game}. Rooted in the dissipativity theory of Willems~\cite{Willems1972Dissipative-dyn}, this approach was further developed in~\cite{Park2019From-Population,Arcak2020Dissipativity-T}, and was extended to convex cones of hybrid rules in~\cite{Certorio_2024aa}. Related passivity ideas have also been used to study reinforcement and higher-order learning in finite games~\cite{Gao2020On-Passivity-Re}. However, important rules with so-called \emph{imitative} behavior~\cite[\S13.1.5]{Sandholm2015Handbook-of-gam}, including the replicator rule, are not $\delta$-passive~\cite{Park2018Passivity-and-e}. Counterclockwise dissipativity (CCW), introduced in~\cite{Angeli2006Systems-With-Co} to study the robust stability of feedback systems, was adopted in~\cite{Martins_2025aa,Hankins_2025aa,Hankins_2026aa} to overcome this limitation and study hierarchical decision making~\cite{YuWenCDC2025}. Namely, CCW provides a passivity-like framework for establishing learning/convergence under a broad convex cone of so-called well-behaved rules, which steer the population state in a direction positively correlated with the payoffs and are stationary exactly at best responses. Among the well-behaved rules are not only the aforementioned imitative rules, but also discontinuous rules such as the best response~\cite{Hankins_2025aa} rule and all noise-free rules known to be $\delta$-passive. 

The main result in~\cite{Martins_2025aa} establishes Nash equilibrium learning by combining an integral inequality with Barb\u{a}lat's lemma. It proves attractivity of the Nash equilibrium set, but provides neither a Lyapunov function nor Lyapunov stability. This is a notable limitation because a Lyapunov function can confine trajectories to invariant sublevel sets, providing bounds on transient excursions and supporting robustness and safety guarantees. It is also useful for relating deterministic evolutionary dynamics to the finite population models they approximate. In particular, the existence of a Lyapunov function can lead to stronger probabilistic guarantees on how closely the finite population state follows the limiting deterministic trajectory~\cite{Benaim2003Deterministic-a}.

For $\delta$-passive payoff mechanisms, Lyapunov functions can be constructed by combining storage functions for the payoff mechanism and the evolutionary dynamics stemming from the rule, as in~\cite{Fox2013Population-Game,Park2019From-Population,Certorio_2024aa}. This construction, however, depends on the rule being $\delta$-passive and on a suitable storage function being available. It therefore does not extend directly to the broader class of rules covered by the CCW framework.

\subsection{Main Objective, Contributions and Paper Structure}

Our objective is to obtain Lyapunov stability and convergence guarantees for the feedback interconnection of a CCW payoff mechanism with an evolutionary dynamics model, when the rule is known only to belong to the convex cone of well-behaved rules considered in~\cite{Martins_2025aa,Hankins_2025aa}. To this end, we adapt the storage function characterization of CCW systems and the Lyapunov construction for positive feedback loops proposed in~\cite{Angeli2006Systems-With-Co}. 

Our main contribution is a rule-independent Lyapunov methodology for CCW payoff mechanisms. A CCW storage function yields a closed-loop Lyapunov function, while the rule enters the analysis only through a nonnegative payoff-alignment term\footnote{This term is formalized later as the correlation function.} shared by all well-behaved rules. Thus, the same Lyapunov function applies throughout the cone without requiring knowledge of the particular rule. We then specialize this methodology to a nonlinear class of anticipatory payoff mechanisms, where agents compare current payoffs with those associated with a smoothed population state. Extending earlier linear formulations~\cite{Shamma2005Dynamic-fictiti,Arslan2006Anticipatory-le,Fox2013Population-Game,Hankins_2025aa}, we show that a Bregman divergence~\cite{Bregman1967A-Relaxation-Me} provides the required CCW storage and leads to a Lyapunov function consisting of a potential gap and a Bregman divergence between the current and smoothed states. This construction establishes global asymptotic stability and is illustrated by a congestion game in which the anticipatory term acts as a nonlinear congestion-growth toll.

\paragraph*{\bf Paper structure} \S\ref{sec:framework} describes the framework and the problem, \S\ref{sec:rulesCCW} introduces well-behaved rules and CCW payoff mechanisms, \S\ref{sec:storage} presents the Lyapunov construction, \S\ref{sec:anticipatory} studies anticipatory payoff mechanisms, \S\ref{sec:example} presents the congestion game and \S\ref{sec:conclusions} concludes the paper.

\paragraph*{\bf Notation} We use $A'$ to denote the transpose of $A$ in $\R^{n \times m}$, $\mathbf{1}$ for the vector of ones, and $[a]_+:=\max\{a,0\}$. The gradient and the Hessian of a twice differentiable function $\phi$ are $\nabla \phi$ and $\nabla^2 \phi$. We write $\dot{\phi}(t)$ for the time derivative of a function of time and $\dist(z,\Eq)$ for the Euclidean distance from a point $z$ to a set $\Eq$.

\section{Framework and Problem Formulation}
\label{sec:framework}

We consider a large population of agents, each following one strategy from the set $\{1,\ldots,n\}$, with $n \geq 2$, and each able to revise its strategy repeatedly. The payoff vector $p(t)$ in $\R^n$ collects the net rewards of the strategies at time $t$, and the agents typically seek strategies with higher payoffs. In the large population limit~\cite{Sandholm2010Population-Game}, the aggregate strategic profile is the population state $x(t)$ in
\begin{equation*}
\X := \Big\{ x \in [0,1]^n \ \big| \ \textstyle\sum_{\ell=1}^n x_\ell = 1 \Big\},
\end{equation*}
whose $i$-th entry $x_i(t)$ is the proportion of agents following strategy $i$ at time $t$. The set $\X$ is compact and convex, its tangent space is $\TX := \{v \in \R^n \ | \ \mathbf{1}'v = 0\}$, and any difference of two elements of $\X$ lies in $\TX$.

\subsection{Learning Rules and Evolutionary Dynamics}

A locally Lipschitz continuous map $\cT: \X \times \R^n \to \R^{n \times n}_{\geq 0}$, referred to as a learning rule (rule for short), models the strategy revision preferences of the agents. Namely, $\cT_{ij}(x,p)$ is the rate at which an agent following $i$ switches to $j$. Following~\cite[\S4.1.2]{Sandholm2010Population-Game}, the population state evolves according to the evolutionary dynamics model (EDM)
\begin{subequations}
\begin{align}
\dot{x}(t) &= \cV\big(x(t),p(t)\big), \quad t \geq 0, \tag{EDMa} \label{eq:EDMa}\\
\cV_i(x,p) &:= \sum_{j=1}^n x_j \cT_{ji}(x,p) - \sum_{j=1}^n x_i \cT_{ij}(x,p), \tag{EDMb} \label{eq:EDMb}
\end{align}
\end{subequations}
where $\cV_i$ is the net flow of agents into strategy $i$. Since $\mathbf{1}'\cV(x,p)=0$ and $\cV_i(x,p) \geq 0$ whenever $x_i = 0$, the set $\X$ is forward invariant. See~\cite[Part~II]{Sandholm2010Population-Game} and~\cite[\S13.3--13.5]{Sandholm2015Handbook-of-gam} for a discussion of rule types. 

\begin{remark}[\bf Rule continuity] For simplicity, in this paper, we consider Lipschitz continuous rules. However, we could have allowed for discontinuous rules, such as the best response rule, by considering more general notions of solution and appropriately changing the way we handle time derivatives along solutions as done in \cite{Hankins_2025aa}. 
Furthermore, \cite[Example~3]{Martins_2025aa} shows that the best response rule can be approximated to an arbitrary accuracy using Lipschitz continuous rules such as those considered in this paper.
\end{remark} 

\subsection{Payoff Mechanisms in State-Space Form}

As in~\cite{Martins_2025aa}, it suffices to consider Lipschitz continuous population state and payoff trajectories, leading to the sets
\begin{align*}
\sX &:= \big\{ \mathbf{x}:[0,\infty) \to \X \ \big| \ \mathbf{x} \text{ is Lipschitz continuous} \big\}, \\
\sP &:= \big\{ \mathbf{p}:[0,\infty) \to \R^n \ \big| \ \mathbf{p} \text{ is Lipschitz continuous} \big\}.
\end{align*}
A payoff mechanism $\PM:\sX \to \sP$ is a causal map that assigns a payoff trajectory $\mathbf{p}$ to each population state trajectory $\mathbf{x}$. For the payoff mechanism, $\mathbf{x}$ is the input and $\mathbf{p}$ is the output, while in the EDM the roles are reversed, as illustrated in Fig.~\ref{fig:closedloop}. Since our goal is to construct Lyapunov functions, we focus on payoff mechanisms admitting a state-space representation
\begin{subequations}
\begin{align}
\dot{q}(t) &= \cG\big(x(t),q(t)\big), \quad q(0) = q_0 \in \Q, \tag{PMa} \label{eq:PMa} \\
p(t) &= \cH\big(x(t),q(t)\big), \tag{PMb} \label{eq:PMb}
\end{align}
\end{subequations}
where $q(t)$ takes values in a compact set $\Q \subset \R^m$, the map $\cG: \X \times \Q \to \R^m$ is locally Lipschitz continuous and $\cH$ is continuously differentiable on an open set containing $\X \times \Q$. The following standing assumption is the state-space counterpart of the one in~\cite{Martins_2025aa}.

\begin{assumption} \label{ass:PM} The payoff mechanism~\eqref{eq:PMa}--\eqref{eq:PMb} \mbox{satisfies}:
\begin{enumerate}
\item The set $\Q$ is forward invariant for~\eqref{eq:PMa}, that is, if $q_0 \in \Q$ and $\mathbf{x} \in \sX$, then $q(t) \in \Q$ for all $t \geq 0$.
\item There is a Lipschitz continuous map $\cF: \X \to \R^n$, called the stationary game of $\PM$, s.t. for all $\mathbf{x} \in \sX$
\begin{equation*}
\lim_{t \to \infty} \|\dot{x}(t)\| = 0 \implies \lim_{t \to \infty} \big\| p(t) - \cF\big(x(t)\big) \big\| = 0.
\end{equation*}
\end{enumerate}
\end{assumption}

Under Assumption~\ref{ass:PM}.1, the feedback interconnection of~\eqref{eq:PMa}--\eqref{eq:PMb} with any EDM evolves on the compact and forward invariant set $\X \times \Q$ and, since its right-hand side is locally Lipschitz continuous, it has a unique solution defined for all $t \geq 0$ from any initial condition, along which $\mathbf{x} \in \sX$, $\mathbf{p} \in \sP$ and $\mathbf{p}$ is bounded. The stationary game $\mathcal{F}$ approximates the payoff vector when $\mathbf{x}$ slows down or converges, akin to the input-to-output characteristic in~\cite{Angeli2006Systems-With-Co}.

\begin{figure}[t]
\begin{center}
    \vspace{.1in}
\begin{tikzpicture}[scale=0.9, transform shape, node distance=0.4cm,
 block/.style={rectangle, draw, rounded corners, minimum width=3.2cm, minimum height=0.95cm, align=center, thick},
 flow/.style={-{Latex[length=2.5mm,width=1.8mm]}, thick}]
\node[block, fill=teal!8] (edm) at (0,0) {\large $\dot{x}=\cV(x,p)$ \\ (EDM)};
\node[block, fill=orange!8, below=of edm] (pm) {\large $\PM:\mathbf{x} \mapsto \mathbf{p}$ \\ (PM)};
\draw[flow] (pm.west) -- ++(-0.8,0) |- (edm.west) node[midway, left] {\large $\mathbf{p}$};
\draw[flow] (edm.east) -- ++(0.8,0) |- (pm.east) node[midway, right] {\large $\mathbf{x}$};
\end{tikzpicture}
\end{center}
\vspace{-3mm}
\caption{Feedback interconnection of the EDM and the payoff mech. $\PM$.}
\label{fig:closedloop}
\end{figure}
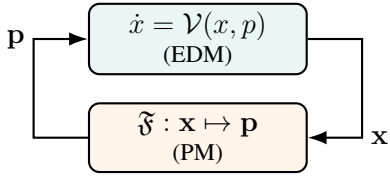

\subsection{Games, Potential Games and Nash Equilibria}

We say that $\PM$ is a game when it is memoryless, that is, $p(t) = \cF\big(x(t)\big)$ for a Lipschitz continuous $\cF:\X \to \R^n$, in which case the stationary game is $\cF$ itself. As in~\cite{Sandholm2001Potential-games}, a game $\cF$ is a potential game when there is a potential $f:\X \to \R$ satisfying, for all $T>0$,
\begin{equation} \label{eq:PotentialGame}
f\big(x(T)\big) - f\big(x(0)\big) = \int_0^T \dot{x}'(t) \cF\big(x(t)\big) dt, \ \ \mathbf{x} \in \sX.
\end{equation}
In particular, $\cF = \nabla f$ for a continuously differentiable $f$ defined on an open set containing $\X$ yields a potential game.

\begin{definition}[\bf Nash equilibria] \label{def:NE} The best response map $\BR:\R^n \to 2^\X$ and the Nash equilibria set of $\cF$ are
\begin{equation*}
\BR(p) := \argmax_{z \in \X}  z'p, \ 
\NE(\cF) := \big\{ x \in \X \ | \ x \in \BR\big(\cF(x)\big) \big\}.
\end{equation*}
\end{definition}

The set $\NE(\cF)$ is nonempty and closed for any continuous $\cF$. If $\cF$ is a potential game with potential $f$, then $\NE(\cF)$ is the set of population states satisfying the Karush-Kuhn-Tucker conditions for maximizing $f$ over $\X$~\cite[\S3, \S5]{Sandholm2001Potential-games}, so that, when $f$ is concave, $\NE(\cF)$ is the set of maximizers.

\subsection{Problem Formulation}

Consider the feedback interconnection in Fig.~\ref{fig:closedloop} with the payoff mechanism~\eqref{eq:PMa}--\eqref{eq:PMb}, where the rule $\cT$ is not known precisely, but is known to belong to a broad convex cone of so-called well-behaved rules we will discuss in \S\ref{sec:rulesCCW} that includes as particular cases most rules studied to date. We seek a continuously differentiable function ${\cL: \X \times \Q \to \R_{\geq 0}}$ that is nonincreasing along the solutions of the closed loop for every well-behaved rule, and that certifies Lyapunov stability of, and convergence to, the set of closed-loop equilibria $(x,q)$ with $x$ in $\NE(\cF)$. 

\section{Well-Behaved Rules and CCW}
\label{sec:rulesCCW}

This section includes some of the definitions and results from~\cite{Martins_2025aa} needed to make this paper self-contained.

\begin{definition} \label{def:PC} Given a rule $\cT$, the correlation function ${\rho: \X \times \R^n \to \R}$ is defined as
\begin{equation} \label{eq:correlation}
\rho(x,p) := p'\cV(x,p), \quad x \in \X, \ p \in \R^n.
\end{equation}
The rule satisfies positive correlation (PC) if, for all $x$ in $\X$ and $p$ in $\R^n$,
\begin{equation*} \tag{PC}
\cV(x,p) \neq 0 \implies \rho(x,p) > 0.
\end{equation*}
It satisfies Nash stationarity if, for all $x \in \X$ and $p \in \R^n$,
\begin{equation*} \tag{NS}
\cV(x,p) = 0 \iff x \in \BR(p).
\end{equation*}
\end{definition}
\begin{definition}[\bf Well-behaved rule] Rules satisfying both (PC) and (NS) are called  well-behaved~\cite{Sandholm2005Excess-payoff-d}.
\end{definition}

Positive correlation is an alignment condition requiring that, away from rest, $\dot{x}$ forms an acute angle with $p$, so that the population moves towards higher average payoff. Under (PC), Nash stationarity is equivalent to
\begin{equation} \label{eq:NSrho}
\rho(x,p) = 0 \iff x \in \BR(p), \quad x \in \X, \ p \in \R^n,
\end{equation}
which is the form we will use. As argued in~\cite[\S III.A]{Martins_2025aa}, (PC) holds for any rule in which the net flow of agents between two strategies is directed towards the one with the higher payoff, which is why most rules of interest satisfy it.

\subsection{Canonical Classes and Hybrid Rules}

The following canonical rule classes are known to satisfy (PC)~\cite{Sandholm2010Pairwise-compar,Sandholm2005Excess-payoff-d,Sandholm2010Population-Game}. In each case the maps $\phi$, $\varphi$ and $\psi$ are locally Lipschitz continuous, and $\hat{p}_j := p_j - x'p$ is the excess payoff of strategy $j$ relative to the population average.

\begin{definition} \label{def:classes} A rule $\cT$ belongs to the pairwise comparison class $\mathbb{T}^{\mathrm{CO}}$ if $\cT_{ij}(x,p) = \phi_{ij}(p)$, where $\phi_{ij}(p)>0$ when $p_j>p_i$ and $\phi_{ij}(p)=0$ otherwise; to the excess payoff class $\mathbb{T}^{\mathrm{EP}}$ if $\cT_{ij}(x,p) = \varphi_j(\hat{p})$, where $\hat{p}'\varphi(\hat{p})>0$ whenever $\hat{p}$ has a positive entry; and to the imitative class $\mathbb{T}^{\mathrm{I}}$ if $\cT_{ij}(x,p) = x_j \psi_{ij}(x,p)$, where $p_j \geq p_i$ if and only if $\psi_{kj}(x,p) - \psi_{jk}(x,p) \geq \psi_{ki}(x,p) - \psi_{ik}(x,p)$ for all $i$, $j$ and $k$.
\end{definition}

\begin{example} \label{ex:rules} Smith's rule~\cite{Smith1984The-stability-o} $\cT^{\mathrm{Smith}}_{ij}(x,p) := [p_j - p_i]_+$ is in $\mathbb{T}^{\mathrm{CO}}$, the Brown-von Neumann-Nash rule~\cite{Brown1950Solutions-of-ga} $\cT^{\mathrm{BNN}}_{ij}(x,p) := [\hat{p}_j]_+$ is in $\mathbb{T}^{\mathrm{EP}}$, and the replicator rule~\cite{Taylor1978Evolutionarily-} $\cT^{\mathrm{R}}_{ij}(x,p) := x_j[p_j - p_i]_+$ is in $\mathbb{T}^{\mathrm{I}}$.
\end{example}

The three classes are convex cones, and since it is unrealistic to expect a population to adhere to a single one of them~\cite[p.~5]{Sandholm2010Pairwise-compar}, we consider the following class.

\begin{definition}[\bf Hybrid rule] \label{def:hybrid} A rule $\cT$ is hybrid if
\begin{equation} \label{eq:hybrid}
\cT = \alpha^{\mathrm{I}}\cT^{\mathrm{I}} + \alpha^{\mathrm{CO}}\cT^{\mathrm{CO}} + \alpha^{\mathrm{EP}}\cT^{\mathrm{EP}} + \tilde{\alpha}\tilde{\cT}, \quad \alpha^{\mathrm{CO}} + \alpha^{\mathrm{EP}} > 0,
\end{equation}
where the weights are nonnegative, $\cT^{\mathrm{I}} \in \mathbb{T}^{\mathrm{I}}$, $\cT^{\mathrm{CO}} \in \mathbb{T}^{\mathrm{CO}}$, $\cT^{\mathrm{EP}} \in \mathbb{T}^{\mathrm{EP}}$ and $\tilde{\cT}$ is any rule satisfying (PC).
\end{definition}

A hybrid rule models agents who, at each revision opportunity, follow one of the component rules with probabilities proportional to the weights, and the constraint ${\alpha^{\mathrm{CO}} + \alpha^{\mathrm{EP}} > 0}$ excludes purely imitative rules, for which unused strategies are never adopted. The following proposition establishes the generality of well-behaved rules.

\begin{proposition}[{\rm see \cite[Proposition~1]{Martins_2025aa}}] \label{prop:hybrid} Any hybrid rule is well-behaved and the correlation function of a hybrid rule is the conic combination of the correlation functions of its components, with the weights in~\eqref{eq:hybrid}.
\end{proposition}

\begin{example} \label{ex:hybrid} The rule $\cT^{\mathrm{b}}_{ij}(x,p) := x_j[p_j - p_i]_+ + 0.01[p_j-p_i]_+$, a conic combination of the replicator and Smith's rules, can be viewed as a close approximation of the replicator rule.
\end{example}

\subsection{CCW Payoff Mechanisms}

We now recall the notion of counterclockwise dissipativity~\cite{Angeli2006Systems-With-Co}, as adapted in~\cite{Martins_2025aa}.

\begin{definition}[\bf CCW] \label{def:CCW} The payoff mechanism $\PM$ with initial state $q_0$ is counterclockwise dissipative, or CCW for short, when
\begin{equation} \label{eq:CCW}
\alpha_\PM := - \inf_{T>0, \ \mathbf{x} \in \sX} \int_0^T \dot{p}'(t) x(t) \, dt < \infty.
\end{equation}
\end{definition}

If a rule satisfies (PC), then $\int_0^T \dot{x}'(t)p(t)dt \geq 0$ for every $\mathbf{p}$ in $\sP$, so the EDM is itself CCW and Fig.~\ref{fig:closedloop} is a positive feedback loop of two CCW systems, the configuration studied in~\cite{Angeli2006Systems-With-Co}. The following are key results in~\cite{Martins_2025aa}.

\begin{theorem}[{\rm see \cite[Theorem~1]{Martins_2025aa}}] \label{thm:potentialCCW} A game is CCW if and only if it is a potential game.
\end{theorem} 
Theorem~\ref{thm:potentialCCW} extends \cite[Proposition 6.2]{Angeli_2009aa}, which establishes that a $C^2$ static map is CCW if and only if its Jacobian is symmetric, to Lipschitz continuous maps on the simplex.

\begin{theorem}[{\rm see \cite[Theorem~2]{Martins_2025aa}}] \label{thm:convergence} If $\cT$ is well-behaved and $\PM$ satisfies Assumption~\ref{ass:PM} and is CCW, then the solution of the feedback system satisfies, for any $x(0)$ in $\X$,
\begin{equation} \label{eq:convergence}
\lim_{t \to \infty} \|\dot{x}(t)\| = 0 \ \text{ and } \ \lim_{t \to \infty} \dist\big(x(t), \NE(\cF)\big) = 0,
\end{equation}
that is, the population achieves evolutionary Nash equilibrium learning.
\end{theorem}

Theorem~\ref{thm:potentialCCW} positions CCW payoff mechanisms as the dynamic extension of potential games, and these mechanisms form a convex cone that includes negative imaginary linear systems~\cite{Lanzon2008Stability-Robus}. Theorem~\ref{thm:convergence} guarantees evolutionary Nash equilibrium learning for every hybrid rule, but, as noted in~\cite{Martins_2025aa}, its proof relies on integration by parts and Barb\u{a}lat's lemma, so~\eqref{eq:convergence} is weaker than the Lyapunov stability available for potential games~\cite[Theorem~4.4]{Sandholm2001Potential-games}. The remainder of the paper removes this tradeoff for payoff mechanisms admitting a CCW storage function.

\section{From CCW Storage to Lyapunov Functions}
\label{sec:storage}

A system is dissipative in the sense of Willems~\cite{Willems1972Dissipative-dyn} with respect to a supply rate when it admits a storage function whose increase along trajectories never exceeds the supplied energy. Following~\cite[Proposition~VI.2]{Angeli2006Systems-With-Co}, the CCW property can be certified in this way, with supply rate $x'\dot{p}$.

\begin{definition}[\bf CCW storage function] \label{def:storage} A CCW storage pair for~\eqref{eq:PMa}--\eqref{eq:PMb} is formed by a continuously differentiable storage function $\cS: \X \times \Q \to \R$ and a continuous dissipation rate $\cD: \X \times \Q \to \R_{\geq 0}$ such that, for every $\mathbf{x}$ in $\sX$ and every $q_0$ in $\Q$, the solution of~\eqref{eq:PMa}--\eqref{eq:PMb} satisfies, for almost every $t \geq 0$,
\begin{equation} \label{eq:storage}
\tfrac{d}{dt} \cS\big(x(t),q(t)\big) \leq x'(t)\dot{p}(t) - \cD\big(x(t),q(t)\big).
\end{equation}
Both sides of~\eqref{eq:storage} exist for almost every $t$ because $\mathbf{x}$ and $\mathbf{q}$ are Lipschitz continuous. Notice that~\eqref{eq:storage} implies that~\eqref{eq:PMa}--\eqref{eq:PMb} is nonlinear negative imaginary as defined in~\cite{Ghallab2018Extending-Negat}.
\end{definition}

\begin{proposition} \label{prop:storageCCW} If~\eqref{eq:PMa}--\eqref{eq:PMb} admits a CCW storage pair $(\cS,\cD)$, then $\PM$ is CCW for every initial state $q_0$ in $\Q$, with $\alpha_\PM \leq \max_{x \in \X} \cS(x,q_0) - \min_{(x,q) \in \X \times \Q} \cS(x,q)$.
\end{proposition}

\paragraph*{\bf Proof} Integrating~\eqref{eq:storage} and using $\cD \geq 0$ gives $\int_0^T \dot{p}'x \, dt \geq \cS\big(x(T),q(T)\big) - \cS\big(x(0),q_0\big)$. The bound follows because $\cS$ is continuous on the compact set $\X \times \Q$.~$\square$

\subsection{Candidate Lyapunov Function}

In~\cite[Theorem~6]{Angeli2006Systems-With-Co}, a candidate Lyapunov function for the positive feedback interconnection of two systems with storage functions $V_1$ and $V_2$ and outputs $h_1$ and $h_2$ is $V_1 + V_2 - h_1'h_2$. In our setting, the EDM plays the role of a CCW system with a trivial storage function and with dissipation rate given by the correlation function, which suggests 
\begin{equation} \label{eq:Lyap}
\cL(x,q) := \cS(x,q) - \cH(x,q)'x + C, \quad (x,q) \in \X \times \Q,
\end{equation}
where $C$ is a constant. The function $\cL$ is continuously differentiable and, as the next lemma shows, its derivative along the closed loop combines the dissipation of the payoff mechanism with the correlation function of the rule.

\begin{lemma} \label{lem:Ldot} Let $(\cS,\cD)$ be a CCW storage pair for~\eqref{eq:PMa}--\eqref{eq:PMb} and let $\cT$ satisfy (PC). Along every solution of the feedback interconnection in Fig.~\ref{fig:closedloop},
\begin{equation} \label{eq:Ldot}
\tfrac{d}{dt}\cL\big(x(t),q(t)\big) \leq - \cD\big(x(t),q(t)\big) - \rho\big(x(t),p(t)\big) \leq 0.
\end{equation}
\end{lemma}

\paragraph*{\bf Proof} Closed-loop solutions are continuously differentiable and $\mathbf{x}$ belongs to $\sX$, so~\eqref{eq:storage} applies. The product rule gives $\frac{d}{dt}\big(p'x\big) = \dot{p}'x + p'\dot{x} = \dot{p}'x + \rho(x,p)$, by~\eqref{eq:EDMa} and~\eqref{eq:correlation}. Subtracting this identity from~\eqref{eq:storage} yields the first inequality, and the second follows from $\cD \geq 0$ and (PC). $\square$

The derivative computation in Lemma~\ref{lem:Ldot} is automatic once a CCW storage pair is available. The delicate part is to select $\cS$ and $C$ so that $\cL$ is nonnegative and vanishes exactly on the equilibria of interest, which we specify in the set
\begin{multline} \label{eq:Eqset}
\Eq := \big\{ (x,q) \in \X \times \Q \ | \ x \in \NE(\cF), \\ \cG(x,q) = 0, \ \cH(x,q) = \cF(x) \big\}.
\end{multline}

\begin{theorem} \label{thm:Lyapunov} Let $\cT$ be well-behaved, and $\PM$ satisfy Assumption~\ref{ass:PM} and admit a CCW storage pair $(\cS,\cD)$ such that:
\begin{itemize}
\item[\textbf{(a)}] for all $(x,q)$ in $\X \times \Q$, $\cD(x,q) = 0$ implies $\cG(x,q) = 0$ and $\cH(x,q) = \cF(x)$;
\item[\textbf{(b)}] there is a constant $C$ such that $\cL$ in~\eqref{eq:Lyap} satisfies $\cL(x,q) \geq 0$ for all $(x,q)$ in $\X \times \Q$, with equality if and only if $(x,q) \in \Eq$.
\end{itemize}
It holds that $\Eq$ is a nonempty compact set of equilibria of the feedback interconnection, and it is globally asymptotically stable. Namely, for every $\epsilon>0$ there is $\delta>0$ such that $\dist\big((x(0),q_0),\Eq\big) < \delta$ implies $\dist\big((x(t),q(t)),\Eq\big) < \epsilon$ for all $t \geq 0$, and $\lim_{t \to \infty}\dist\big((x(t),q(t)),\Eq\big) = 0$ for every initial condition in $\X \times \Q$. In particular,~\eqref{eq:convergence} holds.
\end{theorem}

\paragraph*{\bf Proof} If $(x,q) \in \Eq$ then $\dot{q} = \cG(x,q) = 0$, $p = \cF(x)$ and $x \in \BR(p)$, so $\cV(x,p) = 0$ by (NS). Hence, every point of $\Eq$ is an equilibrium. The set $\Eq$ is closed, because $\NE(\cF)$ is closed and $\cG$, $\cH$ and $\cF$ are continuous, and it is bounded, so it is compact.

\emph{Stability.} Fix $\epsilon>0$ and let $K_\epsilon := \{z \in \X \times \Q \ | \ \dist(z,\Eq) \geq \epsilon\}$. If $K_\epsilon$ is nonempty, it is compact and, by (b), $m := \min_{z \in K_\epsilon} \cL(z) > 0$. Since $\cL$ is continuous and vanishes on the compact set $\Eq$, there is $\delta \in (0,\epsilon)$ such that $\dist(z,\Eq) < \delta$ implies $\cL(z) < m$. By Lemma~\ref{lem:Ldot}, a solution starting at distance less than $\delta$ from $\Eq$ keeps $\cL < m$ and never enters $K_\epsilon$.

\emph{Attractivity.} The set $\X \times \Q$ is compact and forward invariant, and $\cL$ is continuously differentiable with nonpositive derivative along solutions. By LaSalle's invariance principle~\cite{Khalil1995Nonlinear-syste}, every solution converges to the largest invariant set $M$ contained in the set where the derivative of $\cL$ vanishes, which by~\eqref{eq:Ldot} is contained in $Z := \{(x,q) \ | \ \cD(x,q) = 0 \text{ and } \rho\big(x,\cH(x,q)\big) = 0\}$. If $(x,q) \in Z$, then (a) gives $\cG(x,q) = 0$ and $\cH(x,q) = \cF(x)$, and~\eqref{eq:NSrho} gives $x \in \BR\big(\cF(x)\big)$, that is, $x \in \NE(\cF)$. Hence $Z \subset \Eq$. Since $\Eq$ consists of equilibria, it is invariant and contained in the set where the derivative of $\cL$ vanishes, so $M = \Eq$, which is nonempty because limit sets of solutions in a compact set are nonempty. Finally,~\eqref{eq:convergence} follows from the convergence to a compact set of equilibria contained in $\NE(\cF) \times \Q$ and the continuity of $\cV$ and $\cH$. $\square$

\begin{remark}[\bf Rule independence] \label{rem:ruleindep} The Lyapunov function $\cL$ depends only on the payoff mechanism, so the search for a Lyapunov function reduces to the search for a CCW storage pair. The rule enters~\eqref{eq:Ldot} only through its correlation function $\rho$, which is nonnegative and satisfies~\eqref{eq:NSrho} for every well-behaved rule, and which need not be known. Thus, once a CCW storage pair satisfying (a) and (b) is found, Theorem~\ref{thm:Lyapunov} certifies global asymptotic stability for every hybrid rule, including rules with an arbitrarily large imitative component, as in Example~\ref{ex:hybrid}. In contrast, Lyapunov functions built from $\delta$-passivity~\cite{Fox2013Population-Game,Park2019From-Population,Certorio_2024aa} require a storage function for the rule, which is unavailable for the replicator and other imitative rules~\cite{Park2018Passivity-and-e,Martins_2025aa}.
\end{remark}

\begin{remark}[\bf Potential games] \label{rem:potential} Let $\PM$ be a potential game with potential $f$ and no state, so that $\cH = \cF$. Then $\cS(x) := \fmax - f(x) + \cF(x)'x$, with $\fmax := \max_{z \in \X} f(z)$ and $\cD \equiv 0$, is a CCW storage pair, since~\eqref{eq:PotentialGame} gives $\frac{d}{dt}\cS = -\cF'\dot{x} + \dot{p}'x + p'\dot{x} = \dot{p}'x$. With $C=0$, the candidate~\eqref{eq:Lyap} is $\cL(x) = \fmax - f(x)$, condition (a) holds trivially and (b) holds when $f$ is concave, since then the maximizers of $f$ form $\NE(\cF)$, so Theorem~\ref{thm:Lyapunov} recovers the Lyapunov stability of concave potential games~\cite[Theorem~4.4]{Sandholm2001Potential-games}.
\end{remark}

\section{Anticipatory Action and Bregman Divergence}
\label{sec:anticipatory}

Anticipatory learning, in which agents respond to a forecast of the payoffs rather than to their current payoffs, was proposed in~\cite{Shamma2005Dynamic-fictiti,Arslan2006Anticipatory-le} as a form of lead compensation, and linear anticipatory payoff modifications were shown to preserve $\delta$-passivity in~\cite{Fox2013Population-Game} and CCW in~\cite[Example~6]{Martins_2025aa} and~\cite{Hankins_2025aa}. Here we consider a nonlinear class in which the forecast is built from a smoothed version of $\mathbf{x}$ denoted as $\mathbf{q}$.

\subsection{The Class of Anticipatory Payoff Mechanisms}

Let $\mathbb{U}$ be an open convex set containing $\X$ and let $h, g: \mathbb{U} \to \R$ be twice continuously differentiable, with $H := \nabla h$ and $G := \nabla g$. We consider payoff mechanisms of the form
\begin{subequations}
\begin{align}
\dot{q}(t) &= \lambda\big(x(t),q(t)\big)\big(x(t) - q(t)\big), \quad q(0) = q_0 \in \X, \tag{APMa} \label{eq:APMa}\\
p(t) &= H\big(x(t)\big) - G\big(q(t)\big), \tag{APMb} \label{eq:APMb}
\end{align}
\end{subequations}
where $\lambda: \X \times \X \to \R$ is locally Lipschitz continuous and satisfies $\lambda(x,q) \geq \lambda_{\min} > 0$. The state $q$ is a smoothed population state. In general, one can use $\lambda(x,q)$ to capture state-dependent modulation of $\dot{q}$, such as when $\lambda(x,q) = \lambda_0 + \lambda_1\|x-q\|^2$ models smoothing that speeds up when the discrepancy between $q$ and $x$ is large. Defining ${\cF(x) := H(x) - G(x)}$, the output can be written as
\begin{equation} \label{eq:pdecomp}
p = \underbrace{\cF(x)}_{\pbar} + \underbrace{G(x) - G(q)}_{\ptil},
\end{equation}
where $\ptil$ compares the payoff component $G$ at the current and at the smoothed population states, and vanishes whenever $q = x$. The stationary game is a potential game with potential
\begin{equation} \label{eq:potential}
f := h - g, \quad \cF = \nabla f, \quad \fmax := \max_{z \in \X} f(z),
\end{equation}
because $\frac{d}{dt}f\big(x(t)\big) = \dot{x}'\cF(x)$ for every $\mathbf{x}$ in $\sX$, so that~\eqref{eq:PotentialGame} holds. An important special case is obtained with ${H := (1+k)\cF}$ and $G := k\cF$ for a potential game $\cF = \nabla f$ and a gain $k>0$, in which case
\begin{equation} \label{eq:APMk}
p = \cF(x) + k\big(\cF(x) - \cF(q)\big), \ h = (1+k)f, \ g = kf.
\end{equation}

\begin{lemma} \label{lem:wellposed} The payoff mechanism~\eqref{eq:APMa}--\eqref{eq:APMb} satisfies Assumption~\ref{ass:PM} with $\Q = \X$ and stationary game $\cF = H - G$.
\end{lemma}


\paragraph*{\bf Proof}
Let $\mathbf{x} \in \sX$ and $q_0 \in \X$.
The function $\sigma(t) := \mathbf{1}'q(t) - 1$ satisfies
$\dot{\sigma} = -\lambda\sigma$ and $\sigma(0) = 0$,
hence $\sigma \equiv 0$, and
$\dot{q}_i \geq -\lambda q_i$ gives
$q_i(t) \geq q_i(0)\exp\big(-\int_0^t \lambda \, ds\big) \geq 0$,
so $\X$ is forward invariant.
The mismatch $v := x - q$ satisfies
$\dot{v} = \dot{x} - \lambda v$, so
$\|\dot{x}(t)\| \to 0$ implies $v(t) \to 0$
because $\lambda \geq \lambda_{\min}>0$, and then
$p - \cF(x) = G(x) - G(q) \to 0$
by the uniform continuity of $G$ on $\X$.
$\square$

We now state the assumption under which the anticipatory mechanism admits a CCW storage function. Concavity is understood on $\X$ in the tangential sense.

\begin{assumption} \label{ass:concave} The functions $g$ and $f = h - g$ in~\eqref{eq:APMb} and~\eqref{eq:potential} satisfy:
\begin{enumerate}
\item The function $g$ is strictly concave on $\X$, that is, $v'\nabla^2 g(q) v < 0$ for all $q$ in $\X$ and $v$ in $\TX \setminus \{0\}$.
\item The potential $f$ is concave\footnote{See \S\ref{sec:nonconcave} for how to relax this condition. In the special case~\eqref{eq:APMk}, Assumption~\ref{ass:concave} reduces to strict concavity of $f$ on $\X$.} on $\X$.
\end{enumerate}
\end{assumption}

\subsection{Bregman Divergence}

The Bregman divergence~\cite{Bregman1967A-Relaxation-Me} generated by a differentiable convex function $\psi$ is $\psi(x) - \psi(q) - \nabla\psi(q)'(x-q)$, the gap at $x$ between $\psi$ and its tangent plane at $q$. We use the convex generator $\psi := -g$, which gives
\begin{equation} \label{eq:Bregman}
\Breg(x,q) := g(q) - g(x) + G(q)'(x-q), \quad x, q \in \X.
\end{equation}

\begin{lemma} \label{lem:Bregman} Under Assumption~\ref{ass:concave}.1, the following hold:
\begin{itemize}
\item[\textbf{(a)}] $\Breg(x,q) \geq 0$ for all $x$, $q$ in $\X$, with equality if and only if $x = q$;
\item[\textbf{(b)}] if $\mathbf{x} \in \sX$ and $\mathbf{q}:[0,\infty) \to \X$ is Lipschitz continuous, then for almost every $t \geq 0$
\begin{equation} \label{eq:Bregdot}
\tfrac{d}{dt}\Breg(x,q) = -\big(\underbrace{G(x) - G(q)}_{\tilde{p}}\big)'\dot{x} + \dot{q}'\nabla^2 g(q)(x-q).
\end{equation}
\end{itemize}
\end{lemma}

\paragraph*{\bf Proof} (a) Since $\X$ is convex, the segment from $q$ to $x$ lies in $\X$, and Taylor's formula with integral remainder gives
\begin{equation*}
\Breg(x,q) = -\int_0^1 (1-s)\,(x-q)'\nabla^2 g\big(q + s(x-q)\big)(x-q)\, ds.
\end{equation*}
As $x - q \in \TX$, the integrand is nonpositive and is negative when $x \neq q$ and $0 \leq s < 1$. (b) By the chain rule, $\frac{d}{dt}\big(g(q) - g(x)\big) = G(q)'\dot{q} - G(x)'\dot{x}$ and ${\frac{d}{dt}\big(G(q)'(x-q)\big) = \dot{q}'\nabla^2 g(q)(x-q) + G(q)'(\dot{x} - \dot{q})}$. Adding the two identities gives~\eqref{eq:Bregdot}. $\square$

The first term in~\eqref{eq:Bregdot} is $-\ptil'\dot{x}$, which cancels one of the cross terms that appears when differentiating $\ptil'x$. This cancellation is why the Bregman divergence is the natural energy of the anticipatory term.

\subsection{CCW Storage Function and Lyapunov Function}

\begin{theorem} \label{thm:storageAPM} Let Assumption~\ref{ass:concave}.1 hold and define, with $\pbar$ and $\ptil$ as in~\eqref{eq:pdecomp},
\begin{align}
\cS(x,q) &:= \underbrace{\fmax - f(x) + \pbar'x}_{\text{storage for } \cF} + \underbrace{\Breg(x,q) + \ptil'x}_{\text{storage for } \ptil}, \label{eq:SAPM}\\
\cD(x,q) &:= \lambda(x,q)\,(x-q)'\big(-\nabla^2 g(q)\big)(x-q). \label{eq:DAPM}
\end{align}
Then $(\cS,\cD)$ is a CCW storage pair for~\eqref{eq:APMa}--\eqref{eq:APMb}, and~\eqref{eq:storage} holds with equality. Consequently, the anticipatory payoff mechanism is CCW for every $q_0$ in $\X$.
\end{theorem}

\paragraph*{\bf Proof} Let $\mathbf{x} \in \sX$. Using $\frac{d}{dt}f(x) = \pbar'\dot{x}$, we get
\begin{equation*}
\tfrac{d}{dt}\big(\fmax - f(x) + \pbar'x\big) = -\pbar'\dot{x} + \dot{\pbar}'x + \pbar'\dot{x} = \dot{\pbar}'x.
\end{equation*}
By Lemma~\ref{lem:Bregman}(b) and~\eqref{eq:APMa}, we obtain
\begin{align*}
\tfrac{d}{dt}\big(\Breg(x,q) + \ptil'x\big) &= -\ptil'\dot{x} + \dot{q}'\nabla^2 g(q)(x-q) \\
& \qquad + \dot{\ptil}'x + \ptil'\dot{x} = \dot{\ptil}'x - \cD(x,q).
\end{align*}
Adding the two identities and using $p = \pbar + \ptil$ gives $\frac{d}{dt}\cS = \dot{p}'x - \cD$. Since $x - q \in \TX$, Assumption~\ref{ass:concave}.1 and $\lambda \geq \lambda_{\min} > 0$ give $\cD \geq 0$. The CCW property follows from Proposition~\ref{prop:storageCCW}. $\square$

We can now apply the methodology of \S\ref{sec:storage}. Substituting~\eqref{eq:SAPM} and~\eqref{eq:pdecomp} into~\eqref{eq:Lyap} with $C = 0$, all the cross terms cancel and we obtain the candidate Lyapunov function
\begin{equation} \label{eq:LAPM}
\cL(x,q) = \fmax - f(x) + \Breg(x,q).
\end{equation}

\begin{theorem} \label{thm:mainAPM} Let Assumption~\ref{ass:concave} be satisfied and let $\cT$ be well-behaved. For the feedback interconnection of the EDM with~\eqref{eq:APMa}--\eqref{eq:APMb}, the following hold:
\begin{itemize}
\item[\textbf{(a)}] $\cL(x,q) \geq 0$ for all $(x,q)$ in $\X \times \X$, with equality if and only if $(x,q) \in \Eq = \{(x,x) \ | \ x \in \NE(\cF)\}$;
\item[\textbf{(b)}] along every solution, the derivative $\frac{d}{dt}\cL\big(x(t),q(t)\big)$ equals $-\cD\big(x(t),q(t)\big) - \rho\big(x(t),p(t)\big) \leq 0$;
\item[\textbf{(c)}] the set $\Eq$ is globally asymptotically stable.
\end{itemize}
Further, if $f$ is strictly concave on $\X$, then ${\NE(\cF) = \{x^*\}}$ and $(x^*,x^*)$ is globally asymptotically stable.
\end{theorem}

\paragraph*{\bf Proof} By Assumption~\ref{ass:concave}.2 and the discussion following Definition~\ref{def:NE}, $\fmax - f(x) \geq 0$ with equality if and only if $x \in \NE(\cF)$, which together with Lemma~\ref{lem:Bregman}(a) proves (a). Item (b) follows from Lemma~\ref{lem:Ldot} and the equality in Theorem~\ref{thm:storageAPM}. To prove (c), we verify the conditions of Theorem~\ref{thm:Lyapunov}, whose standing assumptions hold by Lemma~\ref{lem:wellposed} and Theorem~\ref{thm:storageAPM}. If $\cD(x,q) = 0$, then $(x-q)'\nabla^2 g(q)(x-q) = 0$ with $x-q \in \TX$, so $x = q$ by Assumption~\ref{ass:concave}.1, whence $\cG(x,q) = 0$ and $\cH(x,q) = \cF(x) + G(x) - G(q) = \cF(x)$, which is condition (a) of Theorem~\ref{thm:Lyapunov}. Condition (b) is item (a) above, with~\eqref{eq:Eqset} equal to $\{(x,x) \ | \ x \in \NE(\cF)\}$. When $f$ is strictly concave on $\X$, its maximizer is unique. $\square$

\subsection{Nonconcave potentials} \label{sec:nonconcave}

We adopted Assumption~\ref{ass:concave}.2 because it makes the paper easier to read at first and covers important cases. It holds in every congestion game with nondecreasing link latencies, such as the one in~\S\ref{sec:example}, and it follows from Assumption~\ref{ass:concave}.1 in the special case~\eqref{eq:APMk}. The concavity of $f$ is used only to ensure that every Nash equilibrium is a global maximizer of $f$, so that $\cL$ vanishes exactly on $\Eq = \{(x,x) \ | \ x \in \NE(\cF)\}$. Suppose now that Assumption~\ref{ass:concave}.1 holds but some Nash equilibrium is not a global maximizer of $f$, requiring a multi-stability analysis~\cite{Angeli2007Multistability-}. Then, by Lemma~\ref{lem:Bregman} and Theorem~\ref{thm:storageAPM}, $\cL$ is still nonnegative and nonincreasing along solutions, but it vanishes only on the smaller set $\Eq_{\max} = \{(x,x) \ | \ x \in \argmax_{\X} f\}$. Two properties hold in general. First, every solution converges to $\Eq$, even when $\NE(\cF)$ is nonconvex or disconnected, because the attractivity part of the proof of Theorem~\ref{thm:Lyapunov} does not use the concavity of $f$. Second, $\Eq_{\max}$ is Lyapunov stable, because $\cL$ is a Lyapunov function for it. If, in addition, $\NE(\cF)$ is finite, then every solution converges to a single point of $\Eq$, because the limit set of each solution is connected. Moreover, for each Nash equilibrium $\hat{x}$, the point $(\hat{x},\hat{x})$ is asymptotically stable if $\hat{x}$ is a local maximizer of $f$, and it is unstable otherwise. Indeed, if $\hat{x}$ is a local maximizer, then $f(\hat{x}) - f(x) + \Breg(x,q)$ is a local Lyapunov function for $(\hat{x},\hat{x})$, which, together with the first property, gives asymptotic stability. If $\hat{x}$ is not a local maximizer, then a solution with $q_0 = x(0)$ arbitrarily close to $\hat{x}$ and $f\big(x(0)\big) > f(\hat{x})$ cannot converge to $(\hat{x},\hat{x})$, because $\cL$ never increases and starts below its value there, so it converges to another Nash equilibrium. Consequently, $\Eq_{\max}$ is not globally attractive when $f$ has a local maximizer that is not global, and $\Eq$ is not stable when some Nash equilibrium is not a local maximizer of $f$.

\subsection{Further Discussion}


Bregman divergences have been used as storage functions in~\cite{persis_2017aa} and also serve as Lyapunov functions for Hessian Riemannian game dynamics~\cite{Mertikopoulos2018Riemannian-game} and in convergence rate bounds for mirror descent~\cite{Li2025Convergence-rat}, with the generator coming from the learning dynamics rather than, as here, from the payoff mechanism on the other side of the loop.

\paragraph*{\bf Anticipatory action and transient behavior} When $x$ varies slowly compared with $\lambda$, the smoothed state lags behind the population state by about $\dot{x}/\lambda$, so that $\ptil \approx \frac{1}{\lambda}\frac{d}{dt}G\big(x(t)\big)$. In the case~\eqref{eq:APMk}, this gives $p \approx \cF\big(x(t)\big) + \frac{k}{\lambda}\frac{d}{dt}\cF\big(x(t)\big) \approx \cF\big(x(t+k/\lambda)\big)$. Hence, the agents respond to a first-order forecast of the payoffs, in the spirit of~\cite{Arslan2006Anticipatory-le,Fox2013Population-Game,Hankins_2025aa}. After linearization at $x^*$ with constant $\lambda$, the map from $\delta x$ to $\delta \ptil$ has transfer function $\frac{s}{s+\lambda}\nabla^2 g(x^*)$. It is negative imaginary~\cite{Lanzon2008Stability-Robus} because $\nabla^2 g(x^*) \preceq 0$ on $\TX$, consistent with the linear analysis in~\cite{Martins_2025aa}. Theorem~\ref{thm:storageAPM} turns this local frequency-domain property into a global storage identity. Beyond preserving the CCW property, the anticipatory term contributes the dissipation $\cD$ in the derivative $\frac{d}{dt}\cL = -\cD - \rho$ of Theorem~\ref{thm:mainAPM}. This dissipation is active whenever $x$ differs from $q$, and it damps the transient in addition to the dissipation $\rho$ of the rule. Since $\cL$ never increases and $\Breg$ is nonnegative, the following bound holds for every well-behaved rule and all $t \geq 0$,
\begin{equation} \label{eq:transientbound}
\fmax - f\big(x(t)\big) \leq \fmax - f\big(x(0)\big) + \Breg\big(x(0),q_0\big).
\end{equation}

\section{Congestion Game With a Nonlinear Congestion Growth Toll}
\label{sec:example}

We consider the congestion game studied in~\cite[\S V]{Martins_2025aa}, which follows~\cite[Example~13.2]{Sandholm2015Handbook-of-gam} and belongs to a class often used to model traffic assignment~\cite{Beckmann1956Studies-in-the-,Wardrop1952Some-theoretica,Smith1984The-stability-o}. Agents travel from node $A$ to node $D$ in the network of Fig.~\ref{fig:network}, choosing among the routes $ABD$, $ACD$ and $ABCD$, which correspond to strategies $1$, $2$ and $3$. The latency of each link is indicated next to the corresponding edge, and the payoff of a route is the negative of its total delay, which gives the stationary game
\begin{align*}
\cF(x) &= -\big[\,5 + 6(x_1 + x_3) \ \ \ 16x_2^2 + 4 \ \ \ 4 + 6x_1 + 10x_3\,\big]', \\
f(x) &= -5x_1 - 4x_2 - 4x_3 - 3(x_1+x_3)^2 - \tfrac{16}{3}x_2^3 - 2x_3^2,
\end{align*}
where $\NE(\cF) = \{x^*\}$ with $x^* = [1/4 \ \ 1/2 \ \ 1/4]'$.

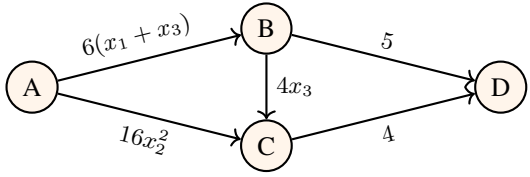
\begin{figure}[t]
\begin{center}
\vspace{.1in}
\begin{tikzpicture}[scale=1.55]
\node[circle, draw, fill=orange!8, thick] (A) at (0,0) {A};
\node[circle, draw, fill=orange!8, thick] (D) at (4,0) {D};
\node[circle, draw, fill=orange!8, thick] (B) at (2,0.5) {B};
\node[circle, draw, fill=orange!8, thick] (C) at (2,-0.5) {C};
\draw[->, thick] (A) -- node[above, xshift=-1mm, rotate=14, font=\small] {$6(x_1+x_3)$} (B);
\draw[->, thick] (A) -- node[below, rotate=-14, font=\small] {$16x_2^2$} (C);
\draw[->, thick] (B) -- node[right, font=\small] {$4x_3$} (C);
\draw[->, thick] (B) -- node[above, rotate=-14, font=\small] {$5$} (D);
\draw[->, thick] (C) -- node[below, rotate=14, font=\small] {$4$} (D);
\end{tikzpicture}
\end{center}
\vspace{-3mm}
\caption{Traffic network with the latency of each link indicated next to the corresponding edge.}
\label{fig:network}
\vspace{-.1 in}
\end{figure}

Drivers have access to smoothed traffic counts $q$, obtained from~\eqref{eq:APMa} with constant $\lambda>0$, and respond to the anticipatory payoff mechanism~\eqref{eq:APMk}:
\begin{equation*} 
\dot{q} = \lambda(x - q), \qquad p = \cF(x) + k\big(\cF(x) - \cF(q)\big).
\end{equation*}
The anticipatory term $\ptil = k\big(\cF(x) - \cF(q)\big)$ is a congestion growth toll, that is, a route on which congestion is growing is charged in proportion to the growth, while a route whose delay is decreasing receives a temporary bonus. The toll vanishes at rest, so it does not alter $\NE(\cF)$.

To verify Assumption~\ref{ass:concave}, we compute for $v$ in $\TX$
\begin{equation*}
v'\big(-\nabla^2 f(q)\big)v = 6(v_1 + v_3)^2 + 4v_3^2 + 32q_2 v_2^2,
\end{equation*}
which vanishes only if $v_1 + v_3 = 0$ and $v_3 = 0$,
hence $v_1 = 0$ and, since $\mathbf{1}'v = 0$, $v_2 = 0$. Thus $f$ is strictly concave on $\X$, including where $q_2 = 0$, and Theorem~\ref{thm:mainAPM} applies with $g = kf$. 
The Bregman divergence in closed form is
\begin{multline} \label{eq:Bregcong}
\Bregf(x,q) = 3\big[(x_1 + x_3) - (q_1 + q_3)\big]^2 + 2(x_3 - q_3)^2 \\ + \tfrac{16}{3}(x_2 - q_2)^2(x_2 + 2q_2),
\end{multline}
The Lyapunov function and its derivative are
\begin{align}
\cL(x,q) &= f(x^*) - f(x) + k\,\Bregf(x,q), \label{eq:Lcong}\\
\tfrac{d}{dt}\cL
&= -k\lambda\big(6(v_1 + v_3)^2 + 4v_3^2 + 32q_2 v_2^2\big) - \rho(x,p), \nonumber
\end{align}
where $v := x - q$, and $(x^*,x^*)$ is globally asymptotically stable
for every well-behaved rule.
Here, $\rho$ need not be known and, for Smith's rule, $\rho(x,p) = \sum_{i,j} x_j[p_i - p_j]_+^2$.

\begin{figure*}[t]
\centering
\vspace{.05in}
\begin{minipage}[t]{0.49\textwidth}
\centering
\input{figures/traj.tikz}
\vspace{-3mm}
\caption{Trajectories for $x(0) \in \{c,d,w,r\}$ and $x^*$ (black circle).}
\label{fig:traj}
\end{minipage} \hfill
\begin{minipage}[t]{0.45\textwidth}
\centering
\hspace{-.3in}
\input{figures/lyap.tikz}
\vspace{-3mm}
\caption{Lyapunov function~\eqref{eq:Lcong} along the trajectories of Fig.~\ref{fig:traj}.}
\label{fig:lyap}
\end{minipage}
\vspace{-.1in}
\end{figure*}

We simulated the closed loop with $k = \lambda = 1$, $q_0 = x(0)$, the rules $\cT^{\mathrm{BNN}}$, $\cT^{\mathrm{Smith}}$ (Example~\ref{ex:rules}) and $\cT^{\mathrm{b}}$ (Example~\ref{ex:hybrid}), and the initial conditions $c := [0 \ 1 \ 0]'$, $d := [0.7 \ 0.3 \ 0]'$, $w := [0 \ 0.2 \ 0.8]'$ and $r := [0.6 \ 0 \ 0.4]'$. Fig.~\ref{fig:traj} shows the trajectories converging to $x^*$, and Fig.~\ref{fig:lyap} shows that $\cL$ decreases monotonically, as guaranteed by Theorem~\ref{thm:mainAPM}.
 
\begin{figure}[t]
\centering
\input{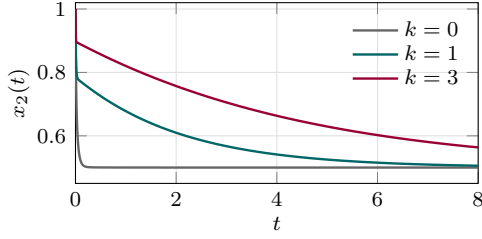}
\vspace{-4mm}
\caption{Proportion $x_2(t)$ of drivers on route $ACD$.}
\label{fig:transient}
\vspace{-.2in}
\end{figure}

Fig.~\ref{fig:transient} illustrates the effect of the toll on the transient. Starting from $c$, where every driver uses route $ACD$, the payoff of route $2$ is $-20$ while routes $1$ and $3$ offer $-5$ and $-4$. Without the toll ($k=0$), the population migrates abruptly and reaches the equilibrium within a fraction of a time unit. With the toll, routes $1$ and $3$ are charged as their delays grow, while route $2$ receives a bonus as it decongests, so the payoffs equalize almost immediately and the migration proceeds gradually, at the pace set by $\lambda$, with the drivers anticipating the congestion that their own switching creates.

\section{Conclusions}
\label{sec:conclusions}

We proposed a methodology to construct Lyapunov functions for evolutionary Nash equilibrium learning under CCW payoff mechanisms. The Lyapunov function is built from a CCW storage function of the payoff mechanism alone. The rule enters only through its correlation function, which appears in the derivative of the Lyapunov function and, for every well-behaved rule, is nonnegative and vanishes only at best responses. Hence, the same Lyapunov function applies to every well-behaved rule, so the guarantees hold without knowledge of the rule and include rules with an arbitrarily large imitative component. For nonlinear payoff mechanisms with anticipatory action, a Bregman divergence generated by the concave payoff component provides the storage function. 

\bibliographystyle{IEEEtran}
\bibliography{MartinsRefs_cleaned}

\end{document}